\documentclass[11pt,twoside]{amsart}
\usepackage{amsmath, amsthm, amscd, amsfonts, amssymb, graphicx, color}
\usepackage[bookmarksnumbered, colorlinks, plainpages]{hyperref}

\newtheorem{thm}{Theorem}[section]
\newtheorem{cor}[thm]{Corollary}
\newtheorem{lem}[thm]{Lemma}
\newtheorem{prop}[thm]{Proposition}
\newtheorem{defn}[thm]{Definition}
\newtheorem{rem}[thm]{\bf{Remark}}

\newtheorem{ex}[thm]{Example}

\numberwithin{equation}{section}
\def\pn{\par\noindent}

\begin{document}

\leftline{ \scriptsize \it Bulletin of the Iranian Mathematical
Society  Vol. {\bf\rm XX} No. X {\rm(}201X{\rm)}, pp XX-XX.}

\vspace{1.3 cm}

\title{Module approximate amenability of Banach algebras}
\author{Hasan Pourmahmood-Aghababa and Abasalt Bodaghi$^*$}

\thanks{{\scriptsize
\hskip -0.4 true cm MSC(2010): Primary: 43A07, 46H25; Secondary: 43A20, 46H20.
Received: 30 April 2009, Accepted: 21 June 2010.\\
$*$Corresponding author
\newline\indent{\scriptsize $\copyright$ 2011 Iranian Mathematical
Society}}}

\maketitle

\begin{center}
Communicated by\;
\end{center}

\begin{abstract}  In the present paper, the concepts of module (uniform) approximate
amenability and contractibility of Banach algebras that are modules over another Banach algebra, are introduced. The general theory is developed and some hereditary properties are given. In analogy with the Banach algebraic approximate amenability, it is shown that module approximate amenability and contractibility are the same properties. It is also shown that module uniform approximate (contractibility) amenability  and module (contractibility, respectively) amenability for commutative  Banach modules are equivalent. Applying these results to $\ell^1(S)$ as an
$\ell^1(E)$-module, for an inverse semigroup $S$ with the set of
idempotents $E$, it is shown that $\ell^1(S)$ is module
approximately amenable (contractible) if and only if it is module
uniformly approximately amenable if and only if $S$ is amenable.
Moreover, $\ell^1(S)^{**}$ is module (uniformly) approximately amenable if and only if a maximal group homomorphic image of $S$ is finite.
\end{abstract}

\vskip 0.2 true cm


\pagestyle{myheadings}
\markboth{\rightline {\scriptsize  Aghababa and Bodaghi}}
         {\leftline{\scriptsize Module approximate amenability of Banach algebras}}

\bigskip
\bigskip


\section{\bf Introduction}
\vskip 0.4 true cm

The concept of approximately amenable Banach algebras was initiated by Ghahramani and Loy in \cite{GhL}. They characterized the structure of approximately (contractible) amenable Banach algebras through several ways and showed that the group algebra
$L^1(G)$ of a locally compact group $G$ is approximately
(contractible) amenable if and only if $G$ is amenable. This fails
to be true for discrete semigroups. Indeed, Gheoregh and Zhang showed that for a bicyclic semigroup $\mathcal C$, the semigroup algebra $\ell^1(\mathcal C)$ is not approximately amenable \cite{Ghz}. On the other hand, for any discrete semigroup $S$, approximate amenability of $\ell^1(S)$ implies amenability of $S$ \cite{GhLZ} (the case of cancellative semigroups had been earlier proved by Bami and Samea
in \cite{bas}). There are similar results for small
categories in \cite{mpo}. Furthermore,  the approximate
amenability of $\ell^1(S)$ of a  Brandat semigroup $S$ over the
group $G$ is investigated in terms of amenability of $G$ in
\cite{bas}.

It is shown in \cite[Theorem 2.1]{GhLZ} that two concepts of
approximate amenability and approximate contractibility for
Banach algebras are equivalent. Some examples of approximately amenable, non-amenable Banach algebras are given in \cite{GhL} to show that two notions of approximate amenability and amenability do not coincide. Also, some examples of
semigroup algebras $\ell^1(S)$ which are approximately amenable
but not amenable are indicated in \cite{dls}. It is well known that every uniformly
approximately (contractible) amenable Banach algebra is (contractible, respectively) amenable and vice versa (\cite[Theorem 4.1]{GhL}) \cite[Theorem 3.1]{GhLZ}.

The concept of module amenability for a class of Banach algebras that are modules over another Banach algebra has been introduced by Amini in \cite{am1}. He showed that for an inverse semigroup $S$ with the set of idempotents $E$, the semigroup algebra $\ell^1(S)$ is module amenable, as a Banach module over $\ell^1(E)$, if and only if $S$ is amenable. In this paper, we define the notions of module (uniform, $w^*$-) approximate amenability and module (uniform) approximate contractibility for a Banach algebra $\mathcal A$ which is also a Banach
$\mathfrak A$-bimodule with compatible actions; see also \cite{YN} for the definition of module approximate amenability. Then we show that module approximate amenability and module approximate contractibility are the same properties, as are module
uniform approximate (contractibility) amenability and module (contractibility) amenability for those Banach algebras that are commutative as an $\mathfrak A$-bimodule. Some concrete examples of Banach algebras to show that module approximate amenability and module amenability are different notions are given. The fact that approximate amenability of the second dual $\mathcal A^{**}$ implies approximate amenability of $\mathcal A$ has been proved by Ghahramani
and Loy in \cite{GhL}. We prove the module version of this result.

Let $S$ be an inverse semigroup with the set of idempotents $E$.
We show that the semigroup algebra $\ell^1(S)$ is
$\ell^1(E)$-module approximately (contractible) amenable if and
only if $S$ is amenable. The same fact is proved in
\cite[Theorem 2]{YN} by using the amenability of the quotient
group $G_S=G/\sim$ of $S$, where $s\sim t$ whenever there exists
$e\in E$ such that $se=te$ (see \cite{dun}). This can be regarded
as the module version (for inverse semigroups) of a result of
Ghahramani and Loy \cite{GhL} asserting that for any locally
compact group $G$, the group algebra $L^1(G)$ is approximately
(contractible) amenable if and only if $G$ is amenable. It is
also shown that $\ell^1(S)^{**}$ is $\ell^1(E)$-module
(uniformly) approximately amenable if and only if an appropriate
group homomorphic image of $S$ is finite (the module version of
\cite[Theorem 3.3]{GhL} for discrete inverse semigroups).


\section{\bf {\bf \em{\bf Notation and Preliminary Results}}}
\vskip 0.4 true cm

We first recall some definitions in the Banach algebras setting.
Let $\mathcal A$ be a Banach algebra, and let $X$ be a Banach $\mathcal A$-bimodule. A bounded linear map $D: \mathcal A \longrightarrow X$ is called a {\it derivation} if
$$D(ab)=D(a)\cdot b+a\cdot D(b) \qquad (a,b \in \mathcal A).$$

For each $x\in X$, we define the map ${\rm{ad}}_x: \mathcal A \longrightarrow X$
by
$${\rm{ad}}_x(a)=a\cdot x-x\cdot a\qquad (a\in \mathcal A).$$

It is easy to see that ${\rm{ad}}_x$ is a derivation. Derivations of this form are
called {\it inner derivations}. A derivation $D: \mathcal A
\longrightarrow X$ is said to be {\it approximately inner} if there exists a net $(x_i)\subseteq X$ such that
$$D(a)=\lim_i(a\cdot x_i-x_i\cdot a)\qquad (a \in \mathcal A).$$

Hence $D$ is approximately inner if it is in the closure of the
set of inner derivations with respect to the strong operator
topology on $B(\mathcal A, X)$, the space of bounded linear operators from $\mathcal A$ to $X$. A Banach algebra $\mathcal A$ is {\it approximately amenable} ({\it contractible}) if every bounded derivation $D: \mathcal A \longrightarrow X^*$ ($D: \mathcal A \longrightarrow X$) is
approximately inner, for each Banach $\mathcal A$-bimodule $X$ \cite{GhL}, where $X^*$ denotes the first dual of $X$ which is a Banach $\mathcal A$-bimodule in the canonical way.

Let ${\mathcal A}$ and ${\mathfrak A}$ be Banach algebras such
that ${\mathcal A}$ is a Banach ${\mathfrak A}$-bimodule with
compatible actions as follows:
$$ \alpha\cdot(ab)=(\alpha\cdot a)b,
\quad(ab)\cdot\alpha=a(b\cdot\alpha) \qquad(a,b \in {\mathcal
A},\alpha\in {\mathfrak A}).$$

Let ${X}$ be a left Banach ${\mathcal A}$-module and a
Banach ${\mathfrak A}$-bimodule with the following compatible actions:
$$\alpha\cdot(a\cdot x)=(\alpha\cdot a)\cdot x,
\,\,a\cdot(\alpha\cdot x)=(a\cdot\alpha)\cdot x, \,\, a\cdot
(x\cdot\alpha)=(a\cdot x)\cdot \alpha,$$
 for all $a \in{\mathcal
A},\alpha\in {\mathfrak A}$ and $x\in{X}.$ Then we say that ${X}$ is a {\it{left Banach ${\mathcal A}$-${\mathfrak A}$-module}}.
Right Banach $\mathcal A$-$\mathfrak A$-modules and (two-sided) Banach $\mathcal A$-$\mathfrak A$-modules are defined similarly.
Moreover, if $\alpha\cdot x=x\cdot\alpha $ for all $\alpha\in
{\mathfrak A}$ and $x\in{ X}$, then $ X $ is called a {\it
commutative} left (right or two-sided) Banach ${\mathcal A}$-${\mathfrak A}$-module. If $X$ is a
(commutative) Banach ${\mathcal A}$-${\mathfrak A}$-module, then so
is $X^*$, where the actions of $\mathcal A$ and ${\mathfrak A}$
on $ X^*$ are defined as usual:
$$\langle
f\cdot\alpha,x\rangle{}=\langle{}f,\alpha\cdot
x\rangle{},\,\,\langle{} f\cdot a,x\rangle{}=\langle{}f,a\cdot
x\rangle,$$
$$\langle\alpha\cdot f,x\rangle{}=\langle{}f,x\cdot\alpha\rangle{},\,\,\langle{}
a\cdot f,x\rangle{}=\langle{}f,x\cdot a\rangle{}\quad (a
\in{\mathcal A},\alpha\in {\mathfrak A},x\in{ X},f \in X^* ).$$

 Note that in general, ${\mathcal A}$ is not an ${\mathcal A}$-${\mathfrak
A}$-module because ${\mathcal A}$ does not satisfy the
compatibility condition $a\cdot(\alpha\cdot
b)=(a\cdot\alpha)\cdot b$ for $\alpha\in {\mathfrak A}, a,b
\in{\mathcal A}$. But if $\mathcal A$ is a commutative
$\mathfrak A$-module and acts on itself by multiplication from
both sides, then it is also a Banach ${\mathcal A}$-${\mathfrak
A}$-module.

Let ${\mathcal A}$ and ${\mathfrak A}$ be as above
and $X$ be a Banach ${\mathcal A}$-$\mathfrak A$-module. A ($\mathfrak A$-){\it
module derivation} is a bounded map $D: \mathcal A \longrightarrow X $
satisfying
$$D(a\pm b)=D(a)\pm D(b), \quad D(ab)=D(a)\cdot b+a\cdot D(b) \qquad
(a,b \in \mathcal A),$$ and
$$D(\alpha\cdot a)=\alpha\cdot D(a), \quad D(a\cdot\alpha)=D(a)\cdot\alpha
\qquad (a \in{\mathcal A},\alpha\in {\mathfrak A}).$$ 

Note that $D: \mathcal A \longrightarrow X $ is bounded if
there exists $M>0$ such that $\| D(a) \| \leq M \| a \| $, for
each $a \in\mathcal A$. Although $D$ is not necessarily linear,
but still its boundedness implies its norm continuity (since $D$
preserves subtraction). Throughout this paper, all module derivations are assumed to be bounded. When $X $ is a commutative $\mathfrak A$-bimodule, each
$x \in X $ defines a module derivation
$$D_x(a)=a\cdot x-x\cdot a \qquad (a \in{\mathcal A}).$$

These are called {\it inner} module derivations.

\begin{defn}\label{def1}  Let $\mathcal A$ be a Banach algebra and an ${\mathfrak
A}$-bimodule with compatible actions. Then
\begin{enumerate}
\item[(i)] {$\mathcal A$ is {\it module approximately amenable} \emph{(}as an ${\mathfrak
A}$-module\emph{)} if for any commutative Banach ${\mathcal
A}$-${\mathfrak A}$-module $X$, each module derivation $D:
\mathcal A \longrightarrow X^*$ is approximately inner;}
\item[(ii)] {$\mathcal A$ is {\it module approximately contractible} \emph{(}as an ${\mathfrak
A}$-module\emph{)} if for any commutative Banach ${\mathcal
A}$-${\mathfrak A}$-module $X$, each module derivation $D:
\mathcal A \longrightarrow X$ is approximately inner.}
\end{enumerate}
\end{defn}
We will use the qualifier {\it uniform} when that convergence of
the net is uniform over the unit ball of $\mathcal A$, and similarly $w^*$ when that convergence is in the appropriate
weak$^*$-topology.

Recall that a left Banach $\mathcal A$-module $X$ is called {\it{left essential}} if the linear span of $\mathcal A \cdot X=\{a\cdot x : a\in \mathcal A, \, x\in X\}$ is dense in $X$. Right essential $\mathcal A$-modules and two-sided essential $\mathcal A$-bimodules are defined similarly.

We remark that if $\mathcal A$ is a left (right) essential
$\mathfrak{A}$-module, then every $\mathfrak A$-module derivation is also
a derivation, in fact, it is linear. For if $a\in \mathcal A$, there is a sequence $(F_n)\subseteq \mathfrak A\cdot \mathcal A$ such that
$\lim_n F_n=a$. Assume that $F_n=\sum_{m=1}^{K_n}
\alpha_{n,m}\cdot a_{n,m}$ for some finite sequences
$(\alpha_{n,m})_{m=1}^{m=K_n}\subseteq \mathfrak{A}$ and
$(a_{n,m})_{m=1}^{m=K_n}\subseteq \mathcal A$. Then for each $\lambda\in
\mathbb{C}$,
\begin{eqnarray*} 
D(\lambda F_n) \!\! \! & = \!\! \! & D\Big(\lambda \sum_{m=1}^{K_n} \alpha_{n,m} \cdot a_{n,m}\Big)= \sum_{m=1}^{K_n} D\big((\lambda\alpha_{n,m}) \cdot a_{n,m}\big) \vspace{0.2cm} \\ & = \!\! \! & \sum_{m=1}^{K_n} (\lambda\alpha_{n,m}) \cdot D(a_{n,m}) =  \sum_{m=1}^{K_n} \lambda D(\alpha_{n,m} \cdot a_{n,m})=  \lambda D(F_n), 
\end{eqnarray*}
and so, by the continuity of $D$, $D(\lambda a)=\lambda D(a)$, if
we assume that $\mathcal A$ is a left essential $\mathfrak{A}$-module.

The following result is the approximate version of \cite[Proposition 2.1]{am1}.

\begin{prop} \label{pro1}
If $\mathcal A$ is approximately amenable and if it is essential as one-sided $\mathfrak A$-module, then it is module approximately contractible.
\end{prop}
\begin{proof}
Let $\mathcal A$ be an essential Banach left (or right) $\mathfrak{A}$-module, $X$ be a commutative Banach ${\mathcal A}$-${\mathfrak A}$-module, and $D:\mathcal A \longrightarrow X$ be a module derivation. By the above discussion, $D$ is a linear derivation. Since every approximately amenable Banach algebra is approximately contractible (\cite[Theorem 2.1]{GhLZ}), we conclude that $D$ is approximately inner.
\end{proof}

 As we will see in Section 3, there are module
approximately (contractible) amenable Banach algebras that are not
approximately amenable which shows that the converse of
above proposition does not hold in general.

Let $X\widehat{\otimes}Y$ denote the projective tensor product of two Banach spaces $X$ and $Y$. Now consider the module projective tensor product ${\mathcal
A}\widehat{\otimes} _{\mathfrak A} {\mathcal A}$ which is isomorphic
to the quotient space $(\mathcal A \widehat{\otimes} \mathcal
A)/{I_{\mathcal A}}$, where $I_{\mathcal A}$ is the closed linear span of $\{ a\cdot\alpha \otimes b-a
\otimes\alpha \cdot b : \alpha\in {\mathfrak A},a,b\in{\mathcal A}\}$; \cite{rie}. Also consider the closed ideal
$J_{\mathcal A}$ of ${\mathcal A}$ generated by elements of the form $
(a\cdot\alpha) b-a(\alpha \cdot b)$ for $ \alpha\in {\mathfrak
A},a,b\in{\mathcal A}$. We shall denote $I_{\mathcal A}$ and
$J_{\mathcal A}$ by $I$ and $J$, respectively, if there is no risk of confusion. 
Then $I$ is an ${\mathcal A}$-submodule and an ${\mathfrak A}$-submodule of
$\mathcal A \widehat{\otimes} \mathcal A$, $J$ is an ${\mathcal A}$-submodule and an ${\mathfrak A}$-submodule of $\mathcal A$, and both of the quotients $\mathcal A \widehat{\otimes}_{\mathfrak A} \mathcal A$ and $\mathcal A/J$ are
${\mathcal A}$-modules and ${\mathfrak A}$-modules. Also,
$\mathcal A/J$ is a Banach $\mathcal A$-${\mathfrak A}$-module
 when ${\mathcal A}$ acts on ${\mathcal
A}/J$ canonically. Let also $\omega_{\mathcal{A}}:
\mathcal{A}\widehat{\otimes}\mathcal{A}\longrightarrow \mathcal{A}$ be the product map, i.e., $\omega_{\mathcal{A}}(a\otimes b)=ab$, and let
$\tilde{\omega}_{\mathcal{A}} : \mathcal{A}\widehat{\otimes}_{\mathfrak{A}}\mathcal{A} =
{(\mathcal{A}\widehat{\otimes}\mathcal{A}})/{{I}}\longrightarrow
{\mathcal{A}}/{{J}}$ be its induced product map, i.e.,
$\tilde{\omega}_{\mathcal{A}}(a\otimes
b+{I})=ab+{J}$.

We denote by $\square$ the first Arens product on $\mathcal{A}^{**}$, the second dual of $\mathcal{A}$. From now on, we assume that $\mathcal A^{**}$ is equipped with the first Arens product. The canonical images of $a\in\mathcal{A}$ and $\mathcal{A}$ in $\mathcal{A}^{**}$ will be denoted by $\hat{a}$ and $\hat{\mathcal{A}}$, respectively.


\section{\bf {\bf \em{\bf Relations between notions of module approximate amenability}}}
\vskip 0.4 true cm


Let $\mathfrak A$ be a non-unital Banach algebra. Then $\mathfrak A^\#=
\mathfrak A\oplus\mathbb{C}$, the unitization of $\mathfrak A$, is a
unital Banach algebra which contains $\mathfrak A$ as a closed ideal. Let $\mathcal A$  be a Banach algebra and a Banach $\mathfrak A$-bimodule with compatible actions. Then $\mathcal A$ is a Banach algebra and a Banach $\mathfrak A^{\#}$-bimodule with compatible actions in the obvious way, i.e., the actions of $\mathfrak A^\#$ on $\mathcal A$ are as follows:
$$(\alpha, \lambda)\cdot a=\alpha \cdot a+\lambda a, \ \  a\cdot(\alpha, \lambda)=a\cdot\alpha+\lambda a \qquad (\lambda\in \mathbb{C}, \, \alpha\in \mathfrak A, \, a\in \mathcal A).$$

Let $\mathcal A$ be a Banach algebra and a Banach $\mathfrak A$-bimodule with compatible actions and let $\mathcal B=(\mathcal A\oplus \mathfrak A^\#, \bullet)$, where the multiplication $\bullet$ is defined through
$$(a, u)\bullet (b, v)=(ab+av+ub, uv) \qquad (a,b\in \mathcal A, \, u,v\in \mathfrak A^\#).$$

Consider the module actions of $\mathfrak A^\#$ on $\mathcal B$ as follows:
$$u\cdot(a,v)=(u\cdot a, uv), \  \ (a,v)\cdot u=(a\cdot u, vu) \qquad (a\in \mathcal A, \, u,v\in \mathfrak A^\#).$$

Then $\mathcal B$ is a unital Banach algebra and a Banach $\mathfrak A^\#$-bimodule with compatible actions.

\begin{thm} \label{t1}
Let $\mathcal A$ be a Banach algebra and a Banach $\mathfrak A$-bimodule with compatible actions. Then the following are equivalent:
\begin{enumerate}
\item[(i)] {$\mathcal A$ is $\mathfrak A^\#$-module \emph{((}$w^*$-\emph{)} approximately\emph{)} amenable;}
\item[(ii)] {$\mathcal B$ is $\mathfrak A^\#$-module \emph{((}$w^*$-\emph{)} approximately\emph{)} amenable.}
\end{enumerate}
If, in addition, $\mathcal A$ is a left or right essential $\mathfrak A$-module, then \emph{(i)} and \emph{(ii)} are equivalent to
\begin{enumerate}
\item[(iii)] {$\mathcal A$ is $\mathfrak A$-module \emph{((}$w^*$-\emph{)} approximately\emph{)} amenable.}
\end{enumerate}
\end{thm}
\begin{proof}
We prove the case of module amenability, other cases are similar.

(i)$\Longrightarrow$(ii) Let $X$ be a commutative Banach $\mathcal B$-$\mathfrak A^\#$-module and $D:\mathcal B \to X^*$ be an $\mathfrak A^\#$-module derivation. Then for $u\in \mathfrak A^\#$, $D(u)=uD(1)=0$ and thus $D$ reduces to a module derivation $D:\mathcal A \to X^*$. Since $X$ is also a commutative Banach $\mathcal A$-$\mathfrak A^\#$-module, by (i), $D$ is inner.

(ii)$\Longrightarrow$(i) Assume that $X$ is a commutative Banach $\mathcal A$-$\mathfrak A^\#$-module. Then $X$ is a commutative Banach $\mathcal B$-$\mathfrak A^\#$-module in the usual way. Now every $\mathfrak A^\#$-module derivation $D:\mathcal A \to X^*$ extends to an $\mathfrak A^\#$-module derivation $\tilde{D}:\mathcal B \to X^*$ defined by $\tilde{D}(a, u)=D(a)$ for all $a\in \mathcal A$, $u\in \mathfrak A^\#$. By the hypothesis, $\tilde{D}$ is inner and thus $D$ is inner. Whence, $\mathcal A$ is $\mathfrak A^{\#}$-module amenable.

Since every $\mathfrak A^\#$-module is also an $\mathfrak A$-module, (iii) implies (i) trivially.

(i)$\Longrightarrow$(iii) Let $X$ be a commutative Banach
$\mathcal A$-$\mathfrak A$-module and $D:\mathcal A \to X^*$ be an $\mathfrak A$-module
derivation. Then $X$ is a commutative Banach $\mathcal A$-$\mathfrak A^\#$-module (as usual). By the essentiality of $\mathcal A$ as a left (right) $\mathfrak A$-module, $D$ is linear and so 

\begin{align*}
D((\alpha, \lambda)\cdot a) &=D(\alpha\cdot
a+\lambda a)=D(\alpha\cdot a)+D(\lambda a)\\
&=\alpha\cdot
D(a)+\lambda D(a)=(\alpha, \lambda)\cdot D(a),
\end{align*}
for all $\lambda
\in \mathbb{C}, \, \alpha\in \mathfrak A, \, a\in \mathcal A$. Similarly, $D$ is
a right $\mathfrak A^\#$-module homomorphism and so $D$ is an $\mathfrak A^\#$-module
derivation, and by (i) is inner. Therefore, $\mathcal A$ is
$\mathfrak A$-module amenable.
\end{proof}

Let $\mathcal A$ be a Banach algebra and $X$ a Banach $\mathcal A$-bimodule.
One should remember that $X$ is $\mathcal A$-pseudo-unital if
$$X=\mathcal A\cdot X\cdot \mathcal A=\{a\cdot x\cdot b : a, b\in \mathcal A, \, x\in X\}.$$

We need the following lemma which is analogous to \cite[Lemma 2.1]{am1}. Since the proof is similar, it is omitted.

\begin{lem}\label{l1}
If $\mathcal A$ has a bounded approximate identity, then it is $\mathfrak A$-module \emph{(}$w^*$-\emph{)} approximately  amenable if and only if every $\mathfrak A$-module derivation $D:\mathcal A\to X^*$ is approximately inner for each $\mathcal A$-pseudo-unital Banach $\mathcal A$-$\mathfrak A$-module $X$.
\end{lem}

\begin{thm} \label{t2}
Let $\mathcal A$ be a Banach algebra and a Banach $\mathfrak A$-bimodule with compatible actions. Let also $\mathcal B\widehat{\otimes}_{\mathfrak A^\#}\mathcal B$ be commutative as an $\mathfrak A^\#$-module. Then the following are equivalent:
\begin{enumerate}
\item[(i)] {$\mathcal B$ is $\mathfrak A^\#$-module \emph{(}$w^*$-\emph{)} approximately amenable;}
\item[(ii)] {There exists a net $(m_i)\subset (\mathcal B\widehat{\otimes}_{\mathfrak A^\#}\mathcal B)^{**}$ such that for all $b\in \mathcal B$, $b\cdot m_i-m_i\cdot b \to 0$ \emph{(}in the $w^*$-topology of $(\mathcal B\widehat{\otimes}_{\mathfrak A^\#}\mathcal B)^{**}$, respectively\emph{)} and $\tilde\omega_{\mathcal B}^{**}(m_i)=1$ for each $i$;}
\item[(iii)] {There exists a net $(m_i)\subset (\mathcal B\widehat{\otimes}_{\mathfrak A^\#}\mathcal B)^{**}$ such that for all $b\in \mathcal B$, $b\cdot m_i-m_i\cdot b \to 0$ and $\tilde\omega_{\mathcal B}^{**}(m_i)\to 1$ \emph{(}in the $w^*$-topology of $(\mathcal B\widehat{\otimes}_{\mathfrak A^\#}\mathcal B)^{**}$ and $(\mathcal B/J_{\mathcal B})^{**}$, respectively\emph{)}.}
\end{enumerate}
\end{thm}
\begin{proof}
(i)$\Longrightarrow$(ii) Consider the $\mathfrak A^\#$-module derivation $D_{1\otimes_{\mathfrak A^{\#}} 1}: \mathcal B\to (\mathcal B\widehat{\otimes}_{\mathfrak A^\#}\mathcal B)^{**}$. Clearly, $D_{1\otimes_{\mathfrak A^{\#}} 1}(\mathcal B)\subseteq \ker \tilde{\omega}_{\mathcal B}^{**}=(\ker\tilde{\omega}_{\mathcal B})^{**}$ and so there is a net $(n_i)\in \ker \tilde{\omega}_{\mathcal B}^{**}$ such that $D_{1\otimes_{\mathfrak A^{\#}} 1}(b)=\lim_i D_{n_i}(b)$ for all $b\in \mathcal B$. Set $m_i=1\otimes_{\mathfrak A^{\#}} 1- n_i$. Then ${\tilde{\omega}}_{\mathcal B}^{**}(m_i)=1$ and for $b\in \mathcal B$, 
$$b\cdot m_i-m_i\cdot b=D_{1\otimes_{\mathfrak A^{\#}} 1}(b)-(b\cdot n_i-n_i\cdot b) \to 0.$$

(ii)$\Longrightarrow$(iii) is trivial.

(iii)$\Longrightarrow$(i) Let $X$ be a commutative Banach $\mathcal B$-$\mathfrak A^{\#}$-module and $D:\mathcal B \to X^*$ be an $\mathfrak A^{\#}$-module derivation. We follow the standard argument. Since $\mathcal B$ is unital, by Lemma \ref{l1}, we may assume that $X$ is pseudo-unital. Note also that $J_{\mathcal B}$ has zero action on $X$ and thus we can assume that $X$ is a commutative Banach $\mathcal B/J_{\mathcal B}$-$\mathfrak A^{\#}$-module.  Now for each $i$ set $f_i(x)=m_i(F_x)$, where for each $x\in X$, the linear functional $F_x: \mathcal B\to (\mathcal B\widehat{\otimes}_{\mathfrak A^\#}\mathcal B)^*$ is defined by $$F_x(a\otimes_{\mathfrak A^{\#}}b)=\langle x, \, a\cdot D(b)\rangle \qquad (a,b \in \mathcal B).$$ 

It can be easily checked that $F_x$ is well defined. Let $(m_{i,j})\subset \mathcal B\widehat{\otimes}_{\mathfrak A^\#}\mathcal B$ be a net converging weak* to $m_i$. For each $a\in \mathcal B$, $x\in X$, we get $$F_{x\cdot a-a\cdot x}(m_{i,j})=(F_x\cdot a-a\cdot F_x)(m_{i,j})+ \langle x, \, \tilde{\omega}_{\mathcal B}(m_{i,j})\cdot D(a)\rangle.$$ 

Thus, for each $a\in \mathcal B$, $x\in X$, we have
$$\begin{array}{ll} \langle x, D_{f_i}(a)\rangle \!\!\! \vspace{0.1cm}&
= f_i(x\cdot a-a\cdot x) \\ \vspace{0.1cm} &
= m_i(F_{x\cdot a - a\cdot x}) \\ \vspace{0.1cm}&
= \lim_j F_{x\cdot a - a\cdot x}(m_{i,j}) \\ \vspace{0.1cm}&
=m_i(F_x\cdot a-a\cdot F_x) + \lim_j \langle x, \, \tilde{\omega}_{\mathcal B}(m_{i,j}) \cdot D(a)\rangle \\ \vspace{0.1cm}&
=\langle F_x, \, a\cdot m_i-m_i\cdot a\rangle + \langle x, \, \tilde{\omega}_{\mathcal B}^{**}(m_{i}) \cdot D(a)\rangle.
\end{array}$$

Therefore,
$$| \langle x, D_{f_i}(a) - D(a)\rangle | \leq \|a\cdot m_i-m_i\cdot a\| \|D\| \|x\|+\| \tilde{\omega}_{\mathcal B}^{**}(m_{i})-1\| \|D(a)\| \|x\|,$$
and thus $D(a)=\lim_i D_{f_i}(a)$. The proof in the case of
$w^*$-approximate amenability is similar.
\end{proof}

The following parallel result for the approximate contractibility can be proved in a similar way, and so we omit its proof.

\begin{thm} \label{t3}
Let $\mathcal A$ be a Banach algebra and a Banach $\mathfrak A$-bimodule with compatible actions. Then the following are equivalent:
\begin{enumerate}
\item[(i)] {$\mathcal A$ is $\mathfrak A^\#$-module approximately contractible;}
\item[(ii)] {$\mathcal B$ is $\mathfrak A^\#$-module approximately contractible;}
\item[(iii)] {There exists a net $(m_i)\subset \mathcal B\widehat{\otimes}_{\mathfrak A^\#}\mathcal B$ such that for all $b\in \mathcal B$, $b\cdot m_i-m_i\cdot b \to 0$ and $\tilde\omega_{\mathcal B}(m_i)\to 1$;}
\item[(iv)] {There exists a net $(m_i)\subset \mathcal B\widehat{\otimes}_{\mathfrak A^\#}\mathcal B$ such that for all $b\in \mathcal B$, $b\cdot m_i-m_i\cdot b \to 0$ and $\tilde\omega_{\mathcal B}(m_i)= 1$.}
\end{enumerate}
\end{thm}

We now proceed to show that module approximate amenability and module approximate contractibility are the same properties.

Recall that the {\it convex hull} of a subset $A$ of a normed space $X$, denoted by ${\rm{co}}(A)$, is the intersection of all convex sets in $X$ that contain $A$.

\begin{thm} \label{t4}
Let $\mathcal A$ be a Banach algebra and a Banach $\mathfrak A$-bimodule with compatible actions. Then the following statements are equivalent:
\begin{enumerate}
\item[(i)] {$\mathcal B$ is $\mathfrak A^\#$-module approximately contractible;}
\item[(ii)] {$\mathcal B$ is $\mathfrak A^\#$-module approximately amenable;}
\item[(iii)] {$\mathcal B$ is $\mathfrak A^\#$-module $w^*$-approximately amenable.}
\end{enumerate}
\end{thm}
\begin{proof}
The implications (i) $\Longrightarrow$ (ii) and (ii) $\Longrightarrow$ (iii) are trivial.

(iii) $\Longrightarrow$ (i) Assume that (iii) holds. By Theorem \ref{t2}, there is a net $(m_i)\subset (\mathcal B\widehat{\otimes}_{\mathfrak A^\#}\mathcal B)^{**}$ such that for all $b\in \mathcal B$, $b\cdot m_i-m_i\cdot b \to 0$ and $\tilde\omega_{\mathcal B}^{**}(m_i)\to 1$ in the $w^*$-topology of $(\mathcal B\widehat{\otimes}_{\mathfrak A^\#}\mathcal B)^{**}$ and $(\mathcal B/J_{\mathcal B})^{**}$, respectively. Take $\epsilon>0$ and finite sets $F\subset \mathcal B$, $G\subset (\mathcal B/J_{\mathcal B})^*$ and $H\subset (\mathcal B\widehat{\otimes}_{\mathfrak A^\#}\mathcal B)^*$. Then there is some $i$ such that
$$|\langle f\cdot b-b\cdot f, \, m_i\rangle|=|\langle f, \, b\cdot m_i -m_i \cdot b\rangle|<\epsilon \qquad (f\in H, \, b\in \mathcal B),$$
and
$$|\langle g, \, \tilde\omega_{\mathcal B}^{**}(m_i) -1\rangle| <\epsilon  \qquad (g\in G).$$

By Goldstien's theorem and the $w^*$-continuity of
$\tilde\omega_{\mathcal B}^{**}$, there is $n\in \mathcal B\widehat{\otimes}_{\mathfrak A^\#}\mathcal B$
such that
$$|\langle f, \, b\cdot n -n \cdot b\rangle| = |\langle f\cdot b-b\cdot f, \, n\rangle|<\epsilon \qquad (f\in H, \, b\in {\mathcal B}),$$
and
$$|\langle \tilde\omega_{\mathcal B}(n) -1, \, g\rangle| <\epsilon  \qquad (g\in G).$$

So there is a net $(n_j)_{j\in \Gamma}\subset \mathcal B\widehat{\otimes}_{\mathfrak A^\#}\mathcal B$ such that
$b\cdot n_j -n_j \cdot b\to 0$ and $\tilde\omega_{\mathcal B}(n_j) \to 1$, weakly in $\mathcal B\widehat{\otimes}_{\mathfrak A^\#}\mathcal B$ and $\mathcal B/J_{\mathcal B}$, respectively.
Now for each finite set $F\subset \mathcal B$, say $F=\{b_1, \ldots, b_k\}$,
$$(b_1\cdot n_j -n_j\cdot b_1, \ldots, b_k\cdot n_j -n_j \cdot b_k, \, \tilde\omega_{\mathcal B}(n_j)-1) \to (0, \ldots, 0),$$
weakly in $(\mathcal B\widehat{\otimes}_{\mathfrak A^\#}\mathcal B)^k \oplus (\mathcal B/J_{\mathcal B})$. Thus
$(0, \ldots, 0)$ belongs to the weak closure of $$V={\rm{co}}\{(b_1\cdot n_j -n_j\cdot b_1, \ldots, b_k\cdot n_j -n_j \cdot b_k, \, \tilde\omega_{\mathcal B}(n_j)-1) : j\in \Gamma\}.$$

 By Mazur's theorem $(0, \ldots, 0)$ belongs to the norm closure of $V$. Hence, for each $\epsilon >0$, there exists $u_{F, \epsilon}\in {\rm{co}}\{n_j\}$ such that
$$\|b\cdot u_{F, \epsilon}-u_{F, \epsilon}\cdot b\|<\epsilon, \quad \|\tilde\omega_{\mathcal B}(u_{F, \epsilon})-1\|<\epsilon \qquad (b\in F).$$ 

Therefore, by Theorem \ref{t3}, $\mathcal B$ is $\mathfrak A^\#$-module approximately contractible.
\end{proof}

Using Theorems \ref{t1}, \ref{t3} and \ref{t4}, similar to \cite[Theorem 2.1]{GhLZ}, we have the following result. 

\begin{cor} \label{c1}
Let $\mathcal A$ be a Banach algebra and a Banach $\mathfrak A$-bimodule with compatible actions. Then the following are equivalent:
\begin{enumerate}
\item[(i)] {$\mathcal A$ is $\mathfrak A^\#$-module approximately contractible;}
\item[(ii)] {$\mathcal A$ is $\mathfrak A^\#$-module approximately amenable;}
\item[(iii)] {$\mathcal A$ is $\mathfrak A^\#$-module $w^*$-approximately amenable.}
\end{enumerate}
If, in addition, $\mathcal A$ is left or right essential as an $\mathfrak A$-module, then the above statements are equivalent to the following three statements:
\begin{enumerate}
\item[(iv)] {$\mathcal A$ is $\mathfrak A$-module approximately contractible;}
\item[(v)] {$\mathcal A$ is $\mathfrak A$-module approximately amenable;}
\item[(vi)] {$\mathcal A$ is $\mathfrak A$-module $w^*$-approximately amenable.}
\end{enumerate}
\end{cor}

Let $X, Y$ be Banach $\mathcal A$-$\mathfrak A$-modules. We say that the map $\phi:
X\to Y$ is a left $\mathcal A$-$\mathfrak A$-module homomorphism if it is an
$\mathfrak A$-bimodule homomorphism and a left $\mathcal A$-module homomorphism,
that is
$$\phi(\alpha\cdot x)=\alpha\cdot\phi(x),\quad
\phi(x\cdot\alpha)=\phi(x)\cdot\alpha,\quad \phi(a\cdot x)=a\cdot
\phi(x),$$
for all $\alpha\in \mathfrak A, a\in \mathcal A$ and  $x\in X$. Right $\mathcal A$-$\mathfrak A$-module homomorphisms and (two-sided) $\mathcal A$-$\mathfrak A$-module homomorphisms are defined in a similar fashion.

In the next proposition we deal with the module approximate amenability of $\mathcal A^{**}$ and $\mathcal A$. Consider the module projective tensor product\\ $\mathcal B^{**}\widehat{\otimes} _{\mathfrak A^\#} \mathcal B^{**}$, that is $(\mathcal B^{**} \widehat{\otimes} \mathcal B^{**})/M$, where $M$ is the closed linear span of $\{ F\cdot u \otimes G-F \otimes u\cdot G : u\in {\mathfrak A}^\#,F,G\in\mathcal B^{**}\}$.
 Also consider $N$ to be
the closed ideal of $\mathcal B^{**}$ generated by elements of the form $(F\cdot u )\square G-F
\square (u\cdot G)$, for all $u\in {\mathfrak A}^\#,F,G\in\mathcal B^{**}$.

\begin{prop}\label{p2}
Let $\mathcal A$ be a Banach $\mathcal A$-$\mathfrak A$-module such that
$\mathcal A\widehat{\otimes}_{\mathfrak A}\mathcal A$ is a commutative Banach
$\mathcal A$-$\mathfrak A$-module. If $\mathcal A^{**}$ is module approximately
amenable, then so is $\mathcal A$.
\end{prop}
\begin{proof} By Theorem \ref{t1}, it is enough to show that $\mathcal B$ is $\mathfrak A^\#$-module approximately amenable. Since $\mathcal A$ is a Banach ${\mathcal A}$-${\mathfrak
A}$-module, one can routinely check that $(\mathcal A^{**},\square)$ is a Banach
$\mathcal A^{**}$-$\mathfrak A^{**}$-module. Thus
$\mathcal B^{**}=\mathcal A^{**}\oplus (\mathfrak A^{**})^\#$ is $(\mathfrak A^{**})^\#$-module
amenable by Theorem \ref{t1}, and so is $\mathfrak A^\#$-module
amenable. The discussion before Definition 3.5 in \cite{abe} shows that
there is a continuous linear mapping
$$\Omega_{\mathfrak A}: (\mathcal B^{**}\widehat{\otimes}_{\mathfrak A^\#}\mathcal B^{**})\cong\mathcal B^{**} \widehat{\otimes} \mathcal B^{**}/M\longrightarrow (\mathcal B\widehat{\otimes}
 \mathcal B)^{**}/I_{\mathcal B}^{\perp \perp}\cong(\mathcal B\widehat{\otimes}_{\mathfrak A^\#}\mathcal B)^{**}$$
  such that for $a,b \in \mathcal B $ and $m \in \mathcal B^{**}\widehat{\otimes} \mathcal B^{**}$ the following
equalities hold:
\begin{enumerate}
\item[(i)] {$ \Omega_{\mathfrak A} (a \otimes b+M)= a \otimes b+I_{\mathcal B}^{\perp
\perp}$;}
\item[(ii)] {$ \Omega_{\mathfrak A}(m+M)\cdot b= \Omega_{\mathfrak
A}(m\cdot b+M)$;}
\item[(iii)] {$ b\cdot\Omega_{\mathfrak A}(m+M)= \Omega_{\mathfrak A}(b\cdot
m+M)$;}
\item[(iv)] {$ \tilde{\omega}_{\mathcal B}^{**}(\Omega_{\mathfrak
A}(m+M))=\lambda \circ \tilde{\omega}_{\mathcal B^{**}}(m+M)$,}
\end{enumerate}
where $\lambda:\mathcal B^{**}/N\longrightarrow
\mathcal B^{**}/J_{\mathcal B}^{\perp\perp}$ defined by
$\lambda(F+N)=F+J_{\mathcal B}^{\perp\perp}$ is a surjective continuous
$\mathcal B$-$\mathfrak A^\#$-module homomorphism. By assumption there is a net
$(\theta_j)\subset(\mathcal B^{**}\widehat{\otimes}_{\mathfrak A^\#}\mathcal B^{**})^{**}$
such that for all $F\in \mathcal B^{**}$, $F\cdot\theta_j-\theta_j\cdot F\rightarrow 0$
and
$F\cdot(\tilde{\omega}_{\mathcal B^{**}})^{**}(\theta_j)=F+N^{\perp\perp}$.
In particular, for each $b\in \mathcal B$ we have
$b\cdot\theta_j-\theta_j\cdot b\rightarrow 0$ and
$b\cdot(\tilde{\omega}_{\mathcal B^{**}})^{**}(\theta_j)=b+N\in
(\mathcal B^{**}/N)^{**}$. The canonical embedding $T:(\mathcal B\widehat{\otimes}
 \mathcal B/I_{\mathcal B})^*\longrightarrow (\mathcal B\widehat{\otimes}
 \mathcal B/I_{\mathcal B})^{***}$ and its adjoint are $\mathcal B$-$\mathfrak A^\#$-module
 homomorphisms, and thus for all $b\in \mathcal B$ we have
$$b\cdot T^*(\Omega^{**}_{\mathfrak A}(\theta_j))-T^*(\Omega^{**}_{\mathfrak A}(\theta_j))\cdot b\longrightarrow 0.$$

Consider an arbitrary and fixed $\theta_j$. By Goldstine's
Theorem there exists a bounded net $(\theta^i_j)$ in
$\mathcal B^{**}\widehat{\otimes}_{\mathfrak A^\#}\mathcal B^{**}$ such that $ \theta^i_j
\stackrel{w^{*}}{\longrightarrow}\theta_j$. Hence
\begin{align*}
b\cdot \tilde{\omega}_{\mathcal B}^{**}T^*(\Omega^{**}_{\mathfrak
A}(\theta_j)) &=b\cdot w^*-\lim_i
\tilde{\omega}_{\mathcal B}^{**}T^*(\Omega^{**}_{\mathfrak
A}(\widehat{\theta^i_j}))\\
&= b\cdot w^*-\lim_i
\tilde{\omega}_{\mathcal B}^{**}T^*(\widehat{\Omega_{\mathfrak
A}(\theta^i_j)}) \\
&=b\cdot w^*-\lim_i \tilde{\omega}_{\mathcal B}^{**}({\Omega_{\mathfrak
A}(\theta^i_j)})\\
&=b\cdot w^*-\lim_i(\lambda \circ
\tilde{\omega}_{\mathcal B^{**}}(\theta^i_j))\\
&=b\cdot w^*-\lim_i(\lambda^{**} \circ
\tilde{\omega}_{\mathcal B^{**}}^{**}(\widehat{\theta^i_j}))\\
&=\lambda^{**}(b\cdot w^*-\lim_i
\tilde{\omega}_{\mathcal B^{**}}^{**}(\widehat{\theta^i_j}))\\
&=\lambda^{**}(b\cdot \tilde{\omega}_{\mathcal B^{**}}^{**}({\theta_j})) \to \lambda^{**}(b+N)
=b+J_{\mathcal B}^{\perp\perp}.
\end{align*}

Therefore, the net $(T^*(\Omega^{**}_{\mathfrak A}(\theta_j)))$
satisfies in conditions of Theorem \ref{t2}.
\end{proof}

We say that the Banach algebra ${\mathfrak A}$ acts trivially on $\mathcal A$ from left (right) if for each $\alpha\in \mathfrak A$ and $a\in \mathcal A$, $\alpha\cdot a=\varphi(\alpha)a$ ($a\cdot\alpha=\varphi(\alpha)a$), where $\varphi$ is a character of ${\mathfrak A}$.

Note that Example 6.1 of \cite{GhL} shows that the two notions `approximate
amenability' and `amenability` do not coincide. This is the spacial
case of module approximate amenability and module amenability
with $\mathfrak A=\mathbb{C}$. So module approximate amenability and
module amenability are different notions. The next theorem
shows that the two notions `approximate module contractibility' and `module
contractibility' do not coincide; see \cite[Theorem 3.7]{HPA}. To achieve our aim, we need the following result which is similar to \cite[Proposition 3.2]{abe}.

\begin{prop}\label{pro3+}
Let $\mathcal A$ be an $\mathfrak A$-module approximately \emph{(}contractible\emph{)} amenable Banach algebra with trivial left action, and let $J_0$ be a closed ideal of $\mathcal A$ such that $J \subseteq J_0$. If $\mathcal A/J_0$ has an identity, then $\mathcal A/J_0$ is approximately \emph{(}contractible\emph{)} amenable.
\end{prop}

Recall that an {\it inverse semigroup} is a semigroup $S$ such that for each $s\in S$ there is a unique element $s^*\in S$ with $ss^*s=s$ and $s^*ss^*=s^*$. Elements of the form $s^*s$ are called {\it idempotents} and the set of all idempotents is denoted by $E$.

Let $S$ be a (discrete) inverse semigroup
with the set of idempotents $E$. We consider the
natural order on $E$ as follows:
$$e\leq d \Longleftrightarrow ed=e \hspace{0.3cm}(e,d \in
E).$$ 

 The subsemigroup $E$ of $S$ is commutative and in fact is a semilattice
\cite[Theorem V.1.2]{how}, and so $ \ell ^{1}(E)$ could be
regarded as a commutative subalgebra of $ \ell ^{1}(S)$. Thus $ \ell ^{1}(S)$
is a Banach algebra and a Banach $ \ell ^{1}(E)$-module with
compatible actions \cite{am1}. We consider the following actions of $ \ell ^{1}(E)$ on $ \ell ^{1}(S)$:
$$\delta_e\cdot\delta_s = \delta_s, \ \delta_s\cdot\delta_e = \delta_{se} =
\delta_s * \delta_e \qquad (s \in S,  e \in E). \eqno{(3.1)}$$
Let $\phi$ be the augmentation character on $\ell ^{1}(E)$, that is, $\phi(\delta_{e})=1$ for each $e\in E$. Then for each
$f=\sum_{e\in E}f(e)\delta_{e}\in\ell ^{1}(E)$ and $g=\sum_{s\in
S}g(s)\delta_{s}\in\ell ^{1}(S)$, we have
\begin{align*}
f\cdot g&=(\sum_{e\in E}f(e)\delta_{e})\cdot(\sum_{s\in
S}g(s)\delta_{s})=\sum_{s\in S,e\in E}f(e)g(s)\delta_{e}\cdot\delta_{s} \\
&=\sum_{s\in S,e\in E}f(e)g(s)\delta_{s} =(\sum_{e\in
E}f(e))(\sum_{s\in S}g(s)\delta_{s})=\phi(f)g.
\end{align*}

Therefore, multiplication from left is trivial.
In this case, the
ideal $J_{\ell ^{1}(S)}$ (or $J$)  is the closed linear span of
$\{\delta_{set}-\delta_{st} : s,t \in S,  e \in E \}.$ We consider an equivalence relation on $S$ as follows:
$$s\approx t \Longleftrightarrow \delta_s-\delta_t \in J \qquad (s,t \in
S).$$

For an inverse semigroup $S$, the quotient semigroup
${S}/{\approx}$ is a discrete group (see \cite{abe, HPA}). Indeed,
${S}/{\approx}$ is homomorphic to the maximal group homomorphic
image $G_S$ of $S$ (see \cite{mn, pou2}). In particular, $S$ is amenable if and only if $S/\approx$
is amenable \cite[Theorem 2.9]{HPA}. As in \cite[Theorem
3.3]{ra1}, we may observe that $\ell ^{1}(S)/J\cong {\ell
^{1}}(S/\approx)$. Therefore, $ \ell ^{1}(E)$ induces actions on
$\ell ^{1}(S)/{J}\cong {\ell^{1}}(S/\approx)$ as usual:
$$\delta_e\cdot(\delta_s+J) = \delta_s+J, \,\,(\delta_s+J)\cdot\delta_e = \delta_{se}+J \qquad (s \in S,  e \in E).$$

With the above actions  ${\ell ^{1}}(S/\approx)$ becomes a
commutative Banach $\ell ^{1}(E)$-bimodule (see the proof of \cite[Theorem 2.9]{HPA}).

We remark that with the actions considered in $(3.1)$, $\mathcal A=\ell^1(S)$ is always a right essential $\ell^1(E)$-module. For if $f\in \ell^1(S)$, we have $$f=\sum_{s\in S} f(s)\delta_s = \sum_{s\in S} f(s) \delta_{s} * \delta_{s^*s}= \sum_{s\in S} f(s) \delta_{s} \cdot \delta_{s^*s},$$ that belongs to the closed linear span of $\ell^1(S)\cdot \ell^1(E)=\{\delta_s \cdot \delta_e : s\in S, e\in E \}$.

\begin{thm}\label{t6}
Let $S$ be an inverse semigroup with the set of idempotents $E$.
Then $\ell^1(S)$ is $\ell^1(E)$-module approximately
(contractible) amenable if and only if $S$ is amenable.
\end{thm}
\begin{proof}
Let $\ell^1(S)$ be module approximately (contractible) amenable.
Then by Proposition \ref{pro3+},
$\ell^1(S)/J\cong\ell^1(S/\approx)$ is  approximately
(contractible) amenable. Since $S/\approx$ is a discrete group,
by \cite[Theorem 3.2]{GhL}, it is amenable and so $S$ is amenable.

Conversely, if $S$ is amenable, then by \cite[Theorem 3.1]{am1} (or \cite[Theorem 2.9]{HPA}) we
conclude that $\ell^1(S)$ is module amenable and hence module
approximately (contractible) amenable.
\end{proof}

Module approximate amenability of the second dual $\ell^1(S)^{**}$ for
an inverse semigroup $S$ is characterized in the following theorem.

\begin{thm} \label{t5}{\it Let $S$ be an
inverse semigroup with the set of idempotents $E$. Then $\ell
^{1}(S)^{**}$ is $\ell^1(E)$-module approximately amenable if and only if
${S}/{\approx}$ is finite.}
\end{thm}
\begin{proof}
Let $\ell^{1}(S)^{**}$ be $\ell^1(E)$-module approximately
amenable. Using the same method applied in the proof of \cite[Theorem
3.4]{abe} for module approximate amenability, it follows from Proposition
\ref{pro3+} that $\ell^{1}(S)^{**}/J^{\perp\perp}\cong
\ell^1(S/\approx)^{**}$ is approximately amenable. Therefore,
$S/\approx$ is finite by \cite[Theorem 3.3]{GhL}. The converse is
clear by \cite[Theorem 3.4]{abe} and \cite[Theorem 2.11]{HPA}.
\end{proof}

\begin{ex}\label{e1}
\emph {(i) Let $\mathcal C$ be the bicyclic semigroup in two generators. Then
$\frac{\mathcal C}{\approx}\cong \mathbb{Z}$, by \cite{abe}, is amenable
and thus $\ell^1(\mathcal C)$ is $\ell^1(E)$-module approximately
(contractible) amenable but not $\ell^1(E)$-module contractible,
by \cite[Theorem 3.7]{HPA}. However, $\ell^1(\mathcal C)$ is not
approximately amenable in the classical sense \cite[Theorem]{Ghz}.
Also, since $\mathcal C/\approx$ is infinite, by Theorem \ref{t5},
${\ell ^{1}}(\mathcal C)^{**}$ is not ${\ell ^{1}}(E)$-module approximately amenable.}

\emph {(ii) Let $G$ be a group and $I$ a non-empty set. Then
for $S=\mathcal{M}$($G, I$), the {\it Brandt inverse semigroup}
corresponding to the group $G$ and the index set $I$, it is shown in
\cite{HPA} that $S/\approx$ is the trivial group. Therefore, ${\ell
^{1}}(S)^{**}$ is module approximately amenable and ${\ell
^{1}}(S)$ is module approximately (contractible) amenable. On the
other hand,  ${\ell ^{1}}(S)$ is approximately amenable if and
only if $I$ is finite and $G$ is amenable \cite[Theorem 4.5]{mpo}.}
\end{ex}


\section{\bf {\bf \em{\bf Module uniform approximate (contractibility) amenability}}}
\vskip 0.4 true cm

For a Banach algebra $\mathcal{A}$, let $\mathcal{A}^{op}$ denote the algebra $\mathcal A$ with reversed multiplication, called the {\it opposite algebra} of $\mathcal A$.  If $\mathcal A$ is a commutative Banach $\mathfrak A$-bimodule it can be easily verified that $\mathcal A^{op}$ is also a commutative Banach $\mathfrak A$-bimodule. Also, it is easy to see that, in this case, $\mathcal A\widehat{\otimes}_{\mathfrak A}\mathcal A^{op}$ is a Banach algebra with the following product rule: $$(a{\otimes}b)\cdot (c{\otimes}d)= (ac){\otimes}(db) \qquad (a,b,c,d\in\mathcal A).$$

Here, $K=\ker\tilde{\omega}_{\mathcal A}$ is an $\mathcal A$-$\mathfrak A$-submodule of $\mathcal A\widehat{\otimes}_{\mathfrak A}\mathcal A$. In fact, $K$ is a left ideal in $\mathcal A\widehat{\otimes}_{\mathfrak A}\mathcal A^{op}$ which is called the {\it diagonal ideal}.

Let $\mathcal A$ be a Banach algebra and a commutative Banach
$\mathfrak A$-bimodule, and let $X$ be a commutative Banach
$\mathcal A$-$\mathfrak A$-module. Then with the action defined by $$(a\otimes
b)\cdot x= a\cdot x \cdot b \qquad (a, b\in
\mathcal A, \, x\in X),$$ $X$ becomes a left Banach
$\mathcal A\widehat{\otimes}\mathcal A^{op}$-module. Now for $a, b\in \mathcal A$,
$\alpha\in\mathfrak A$, $x\in X$, we get
$$\begin{array}{ll} [(a\cdot \alpha\otimes b)-(a\otimes \alpha \cdot b)]\cdot x
& =(a\cdot \alpha)\cdot x \cdot b- a\cdot x\cdot
(\alpha\cdot b) \\ & =a\cdot (\alpha\cdot x)\cdot b- a\cdot (x\cdot
\alpha) \cdot b \\ & = a\cdot (x\cdot \alpha)\cdot b- a\cdot
(x\cdot \alpha)\cdot b \\ & = 0,
\end{array}$$
and so $I$ has zero action on $X$ and thus $X$ is a left Banach
$\mathcal A\widehat{\otimes}_{\mathfrak A}\mathcal A^{op}$-module. Also,
$$\begin{array}{ll} (a\otimes b)\cdot (\alpha\cdot x)-((a\otimes b)\cdot \alpha)\cdot x
& =(a\otimes b)\cdot (\alpha\cdot x)-(a\otimes b\cdot
\alpha)\cdot x \\ & =a\cdot (\alpha\cdot x)\cdot b- a\cdot x\cdot
(b\cdot\alpha) \\ & =a\cdot (x\cdot \alpha)\cdot b- a\cdot x\cdot
(\alpha \cdot b) \\ & = a\cdot (x\cdot \alpha)\cdot b- a\cdot
(x\cdot \alpha)\cdot b \\ & = 0,
\end{array}$$
hence $X$ is a commutative left Banach $\mathcal A\widehat{\otimes}_{\mathfrak A}\mathcal A^{op}$-$\mathfrak A$-module.

\begin{prop} \label{p2}
Let $\mathcal A$ be a commutative Banach $\mathfrak A$-module, and let $X, Y$ be commutative Banach $\mathcal A$-$\mathfrak A$-modules.
\begin{enumerate}
\item[(i)] {If $\phi: X\to Y$ is a bounded $\mathcal A$-$\mathfrak A$-module homomorphism, then it is a left $\mathcal A\widehat{\otimes}_{\mathfrak A}\mathcal A^{op}$-$\mathfrak A$-module homomorphism;}
\item[(ii)] {If $Y$ is an essential $\mathcal A$-bimodule and $\varphi: X^*\to Y^*$ is a bounded right $\mathcal A\widehat{\otimes}_{\mathfrak A}\mathcal A^{op}$-$\mathfrak A$-module homomorphism, then $\varphi$ is an $\mathcal A$-$\mathfrak A$-module homomorphism.}
\end{enumerate}
\end{prop}
\begin{proof}
(i) Let $a,b\in \mathcal A$ and $x\in X$. Then
$$\phi((a\otimes b)\cdot x)=\phi(a\cdot x\cdot b)=a\cdot\phi(x)\cdot b=(a\otimes b)\cdot \phi(x),$$ and so $\phi$ is a left $\mathcal A\widehat{\otimes}\mathcal A^{op}$-module homomorphism. Therefore, $\phi$ is a left $\mathcal A\widehat{\otimes}_{\mathfrak A}\mathcal A^{op}$-$\mathfrak A$-module homomorphism.

(ii) Let $a,b,c\in \mathcal A$, $f\in X^*$, and $y\in Y$. Then
$$\begin{array}{ll} \langle x, (a\cdot f)\cdot (b\otimes c)\rangle & = \langle (b\otimes c)\cdot x, (a\cdot f)\rangle= \langle b\cdot x\cdot c, a\cdot f \rangle \\ &=  \langle b\cdot x\cdot (ca), f\rangle = \langle (b\otimes ca)\cdot x, f\rangle \\ &= \langle  x, f\cdot (b\otimes ca)\rangle, \end{array}$$ and thus $(a\cdot f)\cdot (b\otimes c)=f\cdot (b\otimes ca)$. Now
$$\begin{array}{ll} \langle b\cdot y\cdot c, \varphi(a\cdot f)\rangle & =  \langle (b\otimes c)\cdot y, \varphi(a\cdot f)\rangle =  \langle y, (\varphi(a\cdot f))\cdot (b\otimes c)\rangle \\ & = \langle y, \varphi((a\cdot f)\cdot (b\otimes c))\rangle= \langle y, \varphi(f\cdot (b\otimes ca))\rangle \\ & =  \langle y, \varphi(f)\cdot (b\otimes ca)\rangle =  \langle (b\otimes ca)\cdot y, \varphi(f)\rangle \\ & =  \langle b\cdot y\cdot ca, \varphi(f)\rangle =  \langle (b\cdot y\cdot c)\cdot a, \varphi(f)\rangle  \\ &=   \langle b\cdot y \cdot c, a\cdot \varphi(f)\rangle. \end{array}$$ 

The essentiality of $Y$ implies $\varphi(a\cdot f)=a\cdot \varphi(f)$. Similarly, we can show that $\varphi(f\cdot a)= \varphi(f)\cdot a$. Hence $\varphi$ is an $\mathcal A$-$\mathfrak A$-module homomorphism.
\end{proof}

\begin{defn}
Let $X$, $Y$ and $Z$ be commutative Banach $\mathcal A$-$\mathfrak A$-modules and let $f: X\to Y$ and $g: Y\to Z$ be bounded  $\mathcal A$-$\mathfrak A$-module homomorphisms. We say that the short exact sequence $$0\longrightarrow X \stackrel{f}{\longrightarrow} Y \stackrel{g}{\longrightarrow} Z \longrightarrow 0,$$
{\it splits} as $\mathcal A$-$\mathfrak A$-modules if $g$ has a bounded right inverse that is also an $\mathcal A$-$\mathfrak A$-module homomorphism.
\end{defn}

Let $\mathcal A$ be a Banach algebra and a commutative Banach $\mathfrak A$-module with a bounded approximate identity and set $K=\ker\tilde{\omega}_{\mathcal A}$. Then it can be easily seen that the sequence
$$\Pi^*: 0\longrightarrow \mathcal A^* \stackrel{\tilde{\omega}_{\mathcal A}^*}{\longrightarrow} (\mathcal A\widehat{\otimes}_{\mathfrak A}\mathcal A)^* \stackrel{\iota^*}{\longrightarrow} K^* \longrightarrow 0,$$
and $$\Pi^{op*}: 0\longrightarrow \mathcal A^* \stackrel{\tilde{\omega}_{\mathcal A}^*}{\longrightarrow} (\mathcal A\widehat{\otimes}_{\mathfrak A}\mathcal A^{op})^* \stackrel{\iota^*}{\longrightarrow} K^* \longrightarrow 0,$$
are exact, where $\iota$ denotes the inclusion map $K \hookrightarrow \mathcal A\widehat{\otimes}_{\mathfrak A}\mathcal A$.

An element $M\in (\mathcal A\widehat{\otimes}_{\mathfrak A}\mathcal A)^{**}$ is called a {\it{module virtual diagonal}} if $M\cdot a=a\cdot M$, $a\cdot\tilde{\omega}_{\mathcal A}^{**}(M)=\hat{a}$ for $a\in \mathcal A$; \cite[Definition 2.2]{am1}. The following proposition is the module version of \cite[Theorem 1.3]{CL}.

\begin{prop} \label{p2+}
Let $\mathcal A$ be a Banach algebra and a commutative Banach $\mathfrak A$-bimodule.
Then $\mathcal A$ is module amenable if and only if $\mathcal A$ has a bounded approximate identity and the exact sequence $\Pi^*$ (or $\Pi^{op*}$) splits as $\mathcal A$-$\mathfrak A$-modules.
\end{prop}
\begin{proof}
We follow the standard argument of \cite[Theorem 1.3]{CL}. If $\mathcal A$ is module amenable, then it has a bounded approximate identity by \cite[Proposition 2.2]{am1}. Let $M\in (\mathcal A\widehat{\otimes}_{\mathfrak A}\mathcal A)^{**}$ be a module virtual diagonal for $\mathcal A$ that exists by \cite[Theorem 2.1]{am1}. Define the map $\rho: (\mathcal A\widehat{\otimes}_{\mathfrak A}\mathcal A)^{*} \to \mathcal A^*$ by $$\langle a, \rho(f) \rangle= \langle a\cdot f, M\rangle \qquad (a\in \mathcal A, f\in (\mathcal A\widehat{\otimes}_{\mathfrak A}\mathcal A)^{*}).$$ 

For each $a\in \mathcal A, f\in\mathcal A^{*}$,
$$\begin{array}{ll}
\langle a, \rho\circ \tilde{\omega}_{\mathcal A}^*(f)\rangle \!\!\! & = \langle  a\cdot (f\circ \tilde{\omega}_{\mathcal A}), M\rangle =\lim_i \langle  e_i\otimes e_i+I, a\cdot (f\circ \tilde{\omega}_{\mathcal A}) \rangle \vspace{0.1cm} \\ & =\lim_i \langle  e_i\otimes e_i\cdot a+I, f\circ \tilde{\omega}_{\mathcal A} \rangle = \lim_i \langle  e_i^2\cdot a, f \rangle = \langle a, f \rangle.
\end{array}$$

So $\rho$ is a left inverse for $\tilde{\omega}_{\mathcal A}^*$. Also, clearly, $\rho$ is an $\mathcal A$-$\mathfrak A$-module homomorphism. It follows from \cite[Proposition 1.1]{CL} that the exact sequence $\Pi^*$ splits as $\mathcal A$-$\mathfrak A$-modules.

Conversely, let $\Pi^*$ splits as $\mathcal A$-$\mathfrak A$-modules with a splitting map $\theta: K^*\to (\mathcal A\widehat{\otimes}_{\mathfrak A}\mathcal A)^{*}$ and let $(e_i)$ be a bounded approximate identity for $\mathcal A$. Assume that $E$ is a $w^*$-cluster point of the bounded net $(e_i\otimes e_i+I)$ in $(\mathcal A\widehat{\otimes}_{\mathfrak A}\mathcal A)^{**}$. It is easy to check that $M=\theta^*\circ\tilde\omega^{**}_{\mathcal A}(E)$ is a module virtual diagonal for $\mathcal A$, and thus $\mathcal A$ is module amenable by \cite[Theorem 2.1]{am1}. The assertion for $\Pi^{op*}$ is clear.
\end{proof}

The following theorem characterizes the module amenability of commutative Banach $\mathfrak A$-bimodules in terms of existence of a bounded right approximate identity for $K=\ker\tilde{\omega}_{\mathcal A}$, where $\tilde{\omega}_{\mathcal A}: \mathcal A\widehat{\otimes}_{\mathfrak A}\mathcal A^{op} \to \mathcal A$ is the usual multiplication map.

\begin{thm} \label{t20}
Let $\mathcal A$ be a Banach algebra and a commutative Banach
$\mathfrak A$-bimodule. Then $\mathcal A$ is module amenable if and only if
$\mathcal A$ has a bounded approximate identity and $K=\ker
\tilde{\omega}_{\mathcal A}$ has a bounded right approximate identity.
\end{thm}
\begin{proof}
Let $(e_i)$ be a bounded right approximate identity for $K$. Without loss of generality suppose that $(e_i)$ converges to an element $E\in K^{**}$ in the $w^*$-topology of $K^{**}$. Define the bounded linear map $\rho: K^*\to (\mathcal A\widehat{\otimes}_{\mathfrak A}\mathcal A^{op})^*$ by
$$\langle m, \rho(f)\rangle=\langle E, f\cdot m\rangle \qquad (f\in K^*, \, m\in \mathcal A\widehat{\otimes}_{\mathfrak A}\mathcal A^{op}).$$

 Now for $n\in K$ and $f\in K^*$ we have
$$\langle n, (\iota^*\circ \rho)(f) \rangle= \langle E, f\cdot n\rangle=\lim_i \langle e_i, f\cdot n\rangle= \lim_i \langle n\cdot e_i, f\rangle= \langle n, f\rangle,$$ and so $\iota^*\circ \rho= \text{id}_{K^*}$, where $\iota$ denotes the inclusion map $K \hookrightarrow \mathcal A\widehat{\otimes}_{\mathfrak A}\mathcal A^{op}$. Also for $m,n \in \mathcal A\widehat{\otimes}_{\mathfrak A}\mathcal A^{op}$ and $f\in K^*$,
$$\langle n, \rho(f\cdot m)\rangle=\langle E, f\cdot mn\rangle=\langle mn, \rho(f)\rangle= \langle n, \rho(f)\cdot m\rangle,$$ and thus $\rho$ is a right $\mathcal A\widehat{\otimes}_{\mathfrak A}\mathcal A^{op}$-module homomorphism. Clearly, it is an $\mathfrak A$-bimodule homomorphism too. From \cite[Proposition 1.1]{CL} it follows that there is a bounded right $\mathcal A\widehat{\otimes}_{\mathfrak A}\mathcal A^{op}$-module homomorphism $\theta: (\mathcal A\widehat{\otimes}_{\mathfrak A}\mathcal A)^*\to\mathcal A^*$ which is also an $\mathfrak A$-bimodule homomorphism and $\theta\circ\tilde{\omega}^*_{\mathcal A}=\text {id}_{\mathcal A^*}$. Since $\mathcal A$ has a bounded approximate identity, $\mathcal A$ is an essential $\mathcal A$-bimodule, and so by Proposition \ref{p2}, $\theta$ is an $\mathcal A$-$\mathfrak A$-module homomorphism. Therefore, the exact sequence $\Pi^*$ splits as $\mathcal A$-$\mathfrak A$-modules, and hence $\mathcal A$ is module amenable, by Proposition \ref{p2+}.

Conversely, by Proposition \ref{p2+}, the exact sequence $\Pi^{op*}$ splits. Let $\theta: K^*\to (\mathcal A\widehat{\otimes}_{\mathfrak A}\mathcal A^{op})^*$ be
an $\mathcal A$-$\mathfrak A$-module homomorphism (in particular, a left $\mathcal A\widehat{\otimes}_{\mathfrak A}\mathcal A^{op}$-$\mathfrak A$-module homomorphism, by Proposition \ref{p2})
which is a right inverse for $\iota^*$. Since $\mathcal A$ is commutative as an $\mathfrak A$-bimodule, by \cite[Proposition 2.2]{am1} it has a bounded approximate identity,
say $(e_i)$. Then, $(e_i\otimes e_i+I)$ is a bounded approximate identity for $\mathcal A\widehat{\otimes}_{\mathfrak A}\mathcal A^{op}$. Let $E$ be a $w^*$-cluster
point of the latter net in $(\mathcal A\widehat{\otimes}_{\mathfrak A}\mathcal A^{op})^{**}$. Then $E$ is a right identity for $((\mathcal A\widehat{\otimes}_{\mathfrak A}\mathcal A^{op})^{**}, \Box)$.
Now it can be routinely checked that $\theta^*(E)$ is a right identity for $(K^{**}, \Box)$, hence $K$ has a bounded right approximate identity by the
standard argument.
\end{proof}

The following proposition shows that module uniformly approximately amenable Banach algebras that are commutative as $\mathfrak A$-bimodules, have a bounded approximate identity. The corresponding result concerning module amenability and uniform approximate amenability are obtained in \cite[Proposition 2.2]{am1} and \cite[Theorem 4.2]{GhL}, respectively.

\begin{prop} \label{p3}
Let $\mathcal A$ be a Banach algebra and a commutative Banach $\mathfrak A$-bimodule. If $\mathcal A$ is module uniformly approximately amenable, then it has a bounded approximate identity.
\end{prop}
\begin{proof}
Take $\mathcal A^{**}$ as a Banach $\mathcal A$-bimodule with usual left action and zero right action. Then $\mathcal A^{**}$ is a commutative Banach $\mathcal A$-$\mathfrak A$-module and the inclusion map $j: \mathcal A\to \mathcal A^{**}$ is a module derivation. By the assumption there is a sequence $(E_n)$ in $\mathcal A^{**}$ such that $a\Box E_n \to j(a)=\hat{a}$ uniformly for $\|a\|\leq 1$. Using the $w^*$-continuity of the first Arens product in first variable, it can be easily seen that $F\Box E_n\to F$ uniformly on $\{F\in\mathcal A^{**}: \|F\|\leq 1\}$. Let $\mathcal R_G: \mathcal A^{**}\to \mathcal A^{**}$ denote the right multiplication operator by $G\in \mathcal A^{**}$. Then there is $n\in \mathbb{N}$ such that $\|\mathcal R_{E_n}-\text{id}_{\mathcal A^{**}}\|<1$. So $\mathcal R_{E_n}$ is invertible. By surjectivity there is $E\in \mathcal A^{**}$ such that $E\Box E_n=E_n$. Hence for each $G\in \mathcal A^{**}$, $\mathcal R_{E_n}(G\Box E-G)=0$ and so by injectivity of $\mathcal R_{E_n}$, $E$ is a right identity for $(\mathcal A^{**}, \Box)$. Now by the standard argument, $\mathcal A$ has a bounded right approximate identity. Applying $\mathcal A^{op}$ in place of $\mathcal A$ and \cite[Proposition 2.12]{HPA}, one can find a bounded left approximate identity for $\mathcal A$, and thus a bounded approximate identity for $\mathcal A$.
\end{proof}

Using Theorem \ref{t20} and Proposition \ref{p3}, similar to \cite[Theorem 3.1]{GhLZ}, we have the following theorem that characterizes the module uniform approximate amenability of those Banach algebras that are commutative as an $\mathfrak A$-bimodule. 

\begin{thm} \label{t22}
Let $\mathcal A$ be a Banach algebra and a commutative Banach $\mathfrak A$-bimodule. Then $\mathcal A$ is module uniformly approximately amenable if and only if it is module amenable.
\end{thm}

As it has been discussed in \cite{HPA}, if $\mathcal A$ is a Banach algebra and a Banach $\mathfrak A$-bimodule with compatible actions, then $\mathcal A/J$ is always a Banach $\mathcal A$-$\mathfrak A$-module. The following lemma is analogous to \cite[Lemma 2.8]{HPA} in the case of module uniform approximate amenability. The proof is similar, but we include it. 

\begin{lem} \label{l3}
Let $\mathcal A$ be a Banach algebra and a Banach $\mathfrak A$-bimodule with compatible actions. Then $\mathcal A$ is module uniformly approximately \emph{(}contractible\emph{)} amenable if and only if $\mathcal A/J$ is module uniformly approximately \emph{(}contractible\emph{)} amenable.
\end{lem}
\begin{proof}
Let $X$ be a commutative Banach
$\mathcal{A}$-$\mathfrak{A}$-module and let $D: \mathcal A\to X^*$ be a
module derivation. By commutativity of $X$, we have $J\cdot
X=X\cdot J=0$ and so $X$ is a commutative Banach
$\mathcal A/J$-$\mathfrak{A}$-module. Also, $D$ vanishes on $J$. So, $D$
induces a module derivation $\tilde{D}: \mathcal A/J\to X^*$. Since
$\mathcal A/J$ is module uniformly approximately amenable, $\tilde{D}$ is approximately inner and so is $D$. The other direction is an easy observation.

The proof in the case of module uniform approximate contractibility is similar.
\end{proof}

\begin{cor}\label{c3}
Let $S$ be an inverse semigroup with the set of idempotents $E$. Then $\ell^1(S)$ is $\ell^1(E)$-module uniformly approximately amenable if and only if $S$ is amenable.
\end{cor}
\begin{proof}
Assume that $\ell^1(S)$ is $\ell^1(E)$-module uniformly approximately am-\\enable. Then, by Lemma \ref{l3}, $\ell^1(S)/J
\cong \ell^1(S/\approx)$ is $\ell^1(E)$-module uniformly
approximately amenable. As mentioned in \cite{HPA}, we see that
$\ell^1(S/\approx)$ is a commutative $\ell^1(E)$-module. So, $\ell^1(S/\approx)$ is $\ell^1(E)$-module
amenable by Theorem \ref{t22}. Now \cite[Theorem 2.9]{HPA} implies that $S$ is
amenable. For the other direction, again Theorem 2.9 of \cite{HPA}
implies that $\ell^1(S)$ is $\ell^1(E)$-module amenable and hence
$\ell^1(E)$-module uniformly approximately amenable.
\end{proof}

\begin{rem}
There is another proof for Corollary \ref{c3}: If $\ell^1(S)$
is $\ell^1(E)$-module uniformly approximately amenable, it
is $\ell^1(E)$-module approximately amenable and so $S$ is amenable by
Theorem \ref{t6}. Conversely, if $S$ is amenable, then $\ell^1(S)$ is $\ell^1(E)$-module amenable by \cite[Theorem 3.1]{am1}. Therefore,
$\ell^1(S)$ is $\ell^1(E)$-module uniformly approximately amenable.
\end{rem}

Applying \cite[Theorem 3.4]{abe} and \cite[Theorem 2.11]{HPA}, the following corollary can be proved as Corollary \ref{c3}.

\begin{cor}\label{c5}
Let $S$ be an inverse semigroup with the set of idempotents $E$. Then $\ell^1(S)^{**}$ is $\ell^1(E)$-module uniformly approximately amenable if and only if $S/\approx$ is finite.
\end{cor}

In the following theorem we characterize the module uniform approximate contractibility of commutative Banach $\mathfrak A$-bimodules. Since the proof is similar to the proof of \cite[Theorem 4.1]{GhL}, we omit it.

\begin{thm} \label{t23}
Let $\mathcal A$ be a Banach algebra and a commutative Banach $\mathfrak A$-bimodule. Then $\mathcal A$ is module uniformly approximately contractible if and only if it is module contractible.
\end{thm}

\begin{cor}\label{c4}
Let $S$ be an inverse semigroup with the set of idempotents $E$. Then $\ell^1(S)$ is $\ell^1(E)$-module uniformly approximately contractible if and only if $S/ \approx$ is finite.
\end{cor}
\begin{proof}
Suppose that $\ell^1(S)$ is $\ell^1(E)$-module uniformly approximately contractible. By Lemma \ref{l3}, $\ell^1(S)/J \cong \ell^1(S/\approx)$ is $\ell^1(E)$-module uniformly approximately contractible. Since $\ell^1(S/\approx)$ is a commutative $\ell^1(E)$-bimodule,  $\ell^1(S/\approx)$ is $\ell^1(E)$-module contractible by Theorem \ref{t23}. It follows from \cite[Theorem 3.7]{HPA} that $S/ \approx$ is finite. For the converse, again \cite[Theorem 3.7]{HPA} implies that $\ell^1(S)$ is $\ell^1(E)$-module contractible and hence $\ell^1(E)$-module uniformly approximately contractible.
\end{proof}

\vskip 0.4 true cm

\begin{center}{\textbf{Acknowledgments}}
\end{center}
The authors sincerely thank the anonymous reviewer for his
careful reading, constructive comments and fruitful suggestions
to improve the quality of the first draft.  \\ \\
\vskip 0.4 true cm



\bigskip
\bigskip


{\footnotesize \pn{\bf Hasan Pourmahmood-Aghababa}\; \\ { Department of Mathematics, University of Tabriz, Tabriz, Iran\\
{\tt Email: h\_p\_aghababa@tabrizu.ac.ir, h\_pourmahmood@yahoo.com}\\

{\footnotesize \pn{\bf Abasalt Bodaghi}\; \\ {Department of Mathematics, Garmsar Branch, Islamic Azad University, Garmsar, Iran}\\
{\tt Email: abasalt.bodaghi@gmail.com}\\
\end{document}